\documentclass{amsart}

 \usepackage[colorlinks=true]{hyperref}
\hypersetup{urlcolor=blue, citecolor=red}
\usepackage[smallfootnotes=false]{hypdvips}
\usepackage[mathscr]{eucal}
\usepackage{amssymb}
\usepackage{graphicx}

\newtheorem{theorem}{Theorem}[section]

\newtheorem{proposition}{Proposition}[section]

\newtheorem{definition}[theorem]{Definition}

\newcommand{\ep}{\varepsilon}

\newcommand{\rref}[1]{(\ref{#1})}
\newcommand{\HM}[1]{\mathbb{H}^{{#1}}_{\so}}
\def\beq{\begin{equation}}
\def\feq{\end{equation}}
\def\barray{\begin{array}}
\def\farray{\end{array}}
\def\hnu{\overline{\nu}}
\def\Tc{T_{\tt{c}}}
\def\Ta{T_{\tt{a}}}
\def\ua{u_{\tt{a}}}
\def\Rr{{\mathscr{R}}}
\def\zz{{\scriptscriptstyle{0}}}
\def\ss{{\scriptscriptstyle{\Sigma}}}

\def\so{\ss \zz}

\def\Dv{\mathbb{\DD}'}

\def\LP{\mathfrak{L}}
\def\dive{\mbox{div}}
\def\la{\langle}
\def\ra{\rangle}
\def\leqs{\leqslant}
\def\geqs{\geqslant}
\def\sc{ {\scriptstyle{\bullet} }}
\def\ep{\epsilon}
\def\parn{\par \noindent}
\def\reali{{\bf R}}
\def\complessi{{\bf C}}
\def\interi{{\bf Z}}

\def\To{ {\bf T} }
\def\Td{ {\To}^d }
\def\Tt{ {\To}^3 }
\def\Zd{ \interi^d }
\def\Zt{ \interi^3 }
\def\PPP{{\mathscr P}}
\def\DD{D}
\def\Dd{{\mathcal D}}
\def\vain{\rightarrow}

\makeatletter \@addtoreset{equation}{section}

\makeatother

\title[A posteriori estimates for Euler and NS equations]
      {A posteriori estimates for Euler and Navier-Stokes equations}

\author[C. Morosi, M. Pernici and L. Pizzocchero]{}


 \email{carlo.morosi@polimi.it}
 \email{mario.pernici@infn.it}
 \email{livio.pizzocchero@unimi.it}

\begin{document}
\centerline{\bf{To appear in the Proceedings of the
XIV International Conference on Hyperbolic Problems:}}
\centerline{\bf{Theory, Numerics and Applications (HYP2012: Padova, June 25-29, 2012)}}
\vskip 0.8cm
\noindent
\maketitle

\centerline{\scshape Carlo Morosi }
\medskip
{\footnotesize
 \centerline{Dipartimento di Matematica, Politecnico di Milano}
   \centerline{P.za Leonardo da Vinci 32, I-20133 Milano, Italy}
} 

\medskip

\centerline{\scshape Mario Pernici}
\medskip
{\footnotesize
 \centerline{Istituto Nazionale di Fisica Nucleare, Sezione di Milano}
   \centerline{Via Celoria 16, I-20133 Milano, Italy}
}

\medskip

\centerline{\scshape Livio Pizzocchero}
\medskip
{\footnotesize
 \centerline{Dipartimento di Matematica, Universit\`a di Milano}
   \centerline{Via Cesare Saldini 50, I-20133 Milano, Italy}
   \centerline{and Istituto Nazionale di Fisica Nucleare, Sezione di Milano, Italy}
}

\bigskip


\begin{abstract}
The first two sections of this work review the framework
of \cite{appeul} for approximate solutions of the
incompressible Euler or Navier-Stokes (NS) equations on a torus $\Td$,
in a Sobolev setting.
This approach starts from
an approximate solution $\ua$ of the Euler/NS Cauchy problem
and, analyzing it \emph{a posteriori}, produces
estimates on the interval of existence of
the exact solution $u$ and on the distance between
$u$ and $\ua$. The next two sections
present an application to the Euler Cauchy problem,
where $\ua$ is a Taylor polynomial in the time variable $t$;
a special attention is devoted to the case
$d=3$, with an initial datum for which
Behr, Ne$\check{\mbox{c}}$as and Wu have
conjectured a finite time blowup \cite{Nec}.
These sections combine the general approach of \cite{appeul}
with the computer algebra methods developed in \cite{bnw};
choosing the Behr-Ne$\check{\mbox{c}}$as-Wu datum,
and using for $\ua$ a Taylor polynomial of order 52,
a rigorous lower bound is derived on the interval of existence
of the exact solution $u$, and an estimate is obtained for the $H^3$ Sobolev
distance between $u(t)$ and $\ua(t)$.
\end{abstract}

\section{Preliminaries.}
Throughout this work we fix a space
dimension $d \in \{2,3,...\}$; in the application of section 4 we will
put $d=3$.
For $a, b$ in $\reali^d$ or $\complessi^d$ we put $a \sc b :=
\sum_{r=1}^d a_r b_r$ and $|a| := \sqrt{\overline{a} \sc a}$,
with $\overline{\phantom{x}}$ indicating the complex conjugate. \par
Let us consider the $d$-dimensional torus $\Td :=
(\reali/2 \pi \interi)^d$; we denote with $(e_k)_{k \in \Zd}$ the Fourier basis
made of the functions $e_k : \Td \vain \complessi$, $e_k(x) := (2 \pi)^{-d/2}
e^{i k \sc x}$.
Here and in the sequel,  ``a vector field on $\Td$'' means
``an $\reali^d$-valued distribution on $\Td$'' (see, e.g., \cite{accau}); we
write $\Dv(\Td) \equiv \Dv$ for the space of such distributions. Any $v
\in \Dv$ has a weakly convergent Fourier expansion $v = \sum_{k \in \Zd} v_k e_k$, with
coefficients $v_k \in \complessi^d$ such that $\overline{v_k} = v_{-k}$.
\vfill \eject \noindent
In the sequel $\mathbb{L}^p(\Td) \equiv \mathbb{L}^p$ denotes the space
of $L^p$ vector fields $\Td \vain \reali^d$.
For all  $n \in \reali$ we introduce the Sobolev
space of zero mean, divergence free vector fields of order $n$; this is
\par
\vbox{
\beq \HM{n}(\Td) \equiv \HM{n} := \Big\{ v \in \Dv~|~
\int_{\Td} \! \! v~ d x = 0,~\dive \, v = 0,~
\sqrt{-\Delta}^{\,n} v \in \mathbb{L}^2~ \}  \feq
$$ = \Big\{ v \in \Dv~|~v_0 = 0,~ k \sc v_k = 0~\mbox{for all}~ k,~
\sum_{k \in \Zd \setminus \{0 \}} |k|^{2 n} |v_k|^2 < + \infty~\} $$
}
\noindent
(in the above, $\int_{\Td} v~d x$ indicates the action
of $v$ on the test function $1$ and $\sqrt{-\Delta}^{\,n} v := \sum_{k \in \Zd \setminus \{0 \}}
|k|^{n} v_k e_k$). $\HM{n}$ is a Hilbert space with the inner product
and the norm
\beq \la v | w \ra_n := \la \sqrt{-\Delta}^n v | \sqrt{-\Delta}^n w \ra_{L^2}
= \sum_{k \in \Zd \setminus \{0 \}} |k|^{2 n} \overline{v_k} \sc w_k ,
\quad \| v \|_n := \sqrt{\la v | v \ra_n}~; \feq
if $m \leqs n$ then $\HM{n} \subset \HM{m}$.
\subsection{The bilinear map for the Euler or Navier-Stokes (NS) equations.}
Consider two vector fields $v, w$ on $\Td$ such that
$v \in \mathbb{L}^2$ and $\partial_r w \in \mathbb{L}^2$ for $r=1,...,d$;
then we have a well defined vector field $v \sc \partial w \in \mathbb{L}^1$
of components $(v \sc \partial w)_r := \sum_{s=1}^d v_s \partial_s w_r$;
we can apply to this the Leray projection $\LP$, sending $\Dv$
onto the space of divergence free vectors fields, and form the vector field
\beq
\PPP(v, w) := - \LP(v \sc \partial w)~.
\feq
The bilinear map $\PPP$: $(v, w) \mapsto \PPP(v,w)$, which is
a main character of the incompressible Euler/NS equations, is known to
possess the following properties: \par\noindent
(i) For each $n > d/2$, $\PPP$ is continuous from
$\HM{n} \times \HM{n+1}$ to $\HM{n}$; so, there is a
constant $K_{n d} \equiv K_n$ such that
\beq \| \PPP(v, w) \|_n \leqs K_n \| v \|_n \| w \|_{n+1}
\quad \mbox{for $v \in \HM{n}$, $w \in \HM{n+1}$}~. \label{basic} \feq
(ii) For each $n > d/2+1$, there is a constant $G_{n d} \equiv G_n$
such that
\beq |\la \PPP(v,w) | w \ra_n | \leqs G_n \| v \|_n
\| w \|^{2}_{n} \quad \mbox{for $v \in \HM{n}$, $w \in \HM{n+1}$}~.
\label{katineq} \feq
The result (ii) is due to Kato, see \cite{Kato}.
In papers \cite{cog} \cite{cok}, \rref{basic} and \rref{katineq}
are called the ``basic inequality'' and the ``Kato inequality'', respectively;
in these papers, computable upper and lower bounds are given for the
sharp constants appearing therein. From here to the end of this work,
$K_n$ and $G_n$ are constants fulfilling the previous inequalities
(and not necessarily sharp). From \cite{cog} \cite{cok} we know that
we can take
\beq K_3 = 0.323~, \qquad G_3 = 0.438 \qquad \mbox{if $d=3$}~; \label{k3g3} \feq
these values will be useful in the sequel.
\subsection{The Euler/NS Cauchy problem.} Let us fix a Sobolev order
\beq n \in \big({d \over 2} + 1, + \infty\big)~. \label{propen} \feq
We choose a ``viscosity coefficient'' $\nu \in [0,+\infty)$,
and put
\beq \hnu := \left\{ \barray{ll} 1 & \mbox{if $\nu=0$,} \\
2 & \mbox{if $\nu >0$.} \farray \right. \feq Furthermore, we
choose a ``forcing'' \beq f \in C([0,+\infty), \HM{n})
\label{for}\feq
and an initial datum \beq u_0 \in
\HM{n+\hnu}~. \label{in}\feq
\begin{definition}
\textsl{
The Cauchy problem for the (incompressible) fluid
with viscosity $\nu$, initial datum $u_0$ and forcing $f$ is the following:
\beq \mbox{Find}~
u \in C([0, T), \HM{n+\hnu}) \cap C^1([0,T), \HM{n}) \quad \mbox{such that} \label{cau} \feq
$$ {d u \over d t} = \nu \Delta u + \PPP(u,u) + f~, \qquad u(0) = u_0 $$
(with $T \in (0, + \infty]$, depending on $u$). As usually, we
speak of the ``Euler Cauchy problem'' if $\nu=0$, and
of the ``NS Cauchy problem'' if $\nu >0$}. \par
\end{definition}
It is known \cite{Kat2} that the above Cauchy problem has a unique
maximal (i.e., non extendable) solution; any solution
is a restriction of the maximal one.
\section{Approximate solutions of the Euler/NS Cauchy problem}
We consider again the Cauchy problem \rref{cau}, for
given $n, \nu, f, u_0$ as in the previous section.
The definitions and the theorem that follow are taken from
\cite{appeul}.
\begin{definition}
\textsl{
An approximate solution of problem \rref{cau} is
any map $\ua \in C([0, \Ta), \HM{n+\hnu}) \cap C^1([0,\Ta), \HM{n})$
(with $\Ta \in (0,+\infty]$).
Given such a function, we stipulate (i) (ii)}. \par\noindent
\textsl{
(i) The differential
error of $\ua$ is
\beq {d \ua \over d t} - \nu \Delta \ua - \PPP(\ua,\ua) - f~
\in C([0,\Ta), \HM{n})~;  \feq
the datum error is
\beq \ua(0) - u_0 \in \HM{n+ \hnu}~. \feq
(ii) Let $m \in \reali, m \leqs n$. A differential error estimator of order $m$ for $\ua$
is a function
\beq \ep_m \in C([0,\Ta), [0,+\infty))
\quad \mbox{such that} \feq
$$ \|({d \ua \over d t} - \nu \Delta \ua - \PPP(\ua,\ua) - f)(t)
\|_m \leqs \ep_m(t)~\mbox{~~for $t \in [0,\Ta)$}~. $$
Let $m \in \reali$, $m \leqs n + \hnu$.
A datum error estimator of order $m$ for $\ua$ is a real number
\beq \delta_m \in [0,+\infty) \quad \mbox{such that} \quad \| \ua(0) - u_0 \|_m \leqs \delta_m~;
\feq
a growth estimator of order $m$ for $\ua$ is a function
\beq \Dd_m \in C([0,\Ta), [0,+\infty))
\quad \mbox{such that} \quad \| \ua(t) \|_m \leqs \Dd_m(t)~\mbox{~~for $t \in [0,\Ta)$}~.
\label{din} \feq
In particular $\ep_m(t) :=
\|({d \ua/d t} - \nu \Delta \ua - \PPP(\ua,\ua) - f)(t)\|_m$,
$\delta_m := \| \ua(0) - u_0 \|_m$ and
$\Dd_m(t) := \| \ua(t) \|_m$ will be called the
tautological estimators of order $m$ for the differential error,
the datum error and the growth of $\ua$.}
\end{definition}
From here to the end of the section we consider
an approximate solution $\ua$ of problem \rref{cau}
of domain $[0, \Ta)$; this is assumed to possess
differential, datum error and growth estimators
of orders $n$ or $n+1$, indicated with $\ep_n, \delta_n, \Dd_n,
\Dd_{n+1}$. \par
\begin{definition}
\textsl{
Let $\Rr_n \in C([0,\Tc), [0,+\infty))$,
with $\Tc \in (0,\Ta]$. This function is said to fulfil the
control inequalities if
\beq {d^{+} \Rr_n \over d t} \geqs - \nu \Rr_n
+ (G_n \Dd_n + K_n \Dd_{n+1}) \Rr_n + G_n \Rr^2_n + \ep_n
~\mbox{everywhere on $[0,\Tc)$}, \label{cont1} \feq
\beq \Rr_n(0) \geqs \delta_n~. \label{cont2} \feq
In the above $d^{+}/ d t$ indicates the right, upper Dini derivative: so,
for all $t \in [0,\Tc)$,
$(d^{+} \Rr_n/ d t)(t) := \limsup_{h \vain 0^{+}} [\Rr_n(t+h) - \Rr_n(t)]/h$.
}
\end{definition}
\begin{proposition}
\label{main}
Assume
 there is a function $\Rr_n \in C([0,\Tc), [0,+\infty))$ fulfilling
the control inequalities; consider the maximal solution $u$ of the
Euler/NS Cauchy problem \rref{cau}, and denote its domain with $[0,T)$.
Then
\beq T \geqs \Tc~, \label{tta} \feq
\beq \| u(t) - \ua(t) \|_n \leqs \Rr_n(t) \qquad \mbox{for $t \in [0,\Tc)$}~.
\label{furth} \feq
\end{proposition}
\begin{proof} (Sketch) One introduces the function
$\| u - \ua \|_n : t \in [0,T) \cap [0, \Ta) \mapsto  \| u(t) - \ua(t) \|_n$
and shows that
${d^{+} \| u - \ua \|_n/d t} \leqs$
$- \nu \, \| u - \ua \|_n$ $+ (G_n \Dd_n + K_n \Dd_{n+1})$ $\| u - \ua \|_n +
G_n \| u - \ua \|^2_n + \ep_n$,
(see Lemma 4.2 of \cite{appeul}, greatly indebted to
\cite{Che}); moreover,
$\| u(0) - \ua(0) \|_n \leqs \delta_n$.
From here, from the control inequalities \rref{cont1}
\rref{cont2} and from the \v{C}aplygin comparison lemma
one infers that
$\| u(t) - \ua(t)\|_n \leqs \Rr_n(t)$ for $t \in [0, T)
\cap [0,\Tc)$. Finally, it is $T \geqs \Tc$; in fact, it it were
$T < \Tc$, the previous inequality about
$u, \ua$ and $\Rr_n$ would imply $\limsup_{t \vain T^{-}} \| u(t) \|_n
< + \infty$, a fact contradicting the maximality assumption for $u$.
See \cite{appeul} for more details.
\end{proof}
Paper \cite{appeul} presents some applications of
the previous proposition, dealing with both the Euler case
$\nu=0$ and the NS case $\nu >0$; a special
attention is devoted therein to the approximate solutions
$\ua$ provided by the Galerkin method. \par
In this work we present an application of Proposition \ref{main}
to the Euler case $\nu=0$, choosing for $\ua$ a
polynomial in the time variable $t$.
In the next section we present this procedure
in general, giving the error estimators
for approximate solutions of this kind;
in the last section we apply
the procedure choosing for $u_0$
the so-called Behr-Ne$\check{\mbox{c}}$as-Wu
initial datum.
\section{Polynomial approximate solutions for the Euler equations}
Let us recall that $n \in (d/2+1,+\infty)$, and consider
the Euler Cauchy problem with a datum $u_0 \in \HM{n+1}$
and zero external forcing:
\beq \mbox{Find}~
u \in C([0, T), \HM{n+1}) \cap C^1([0,T), \HM{n}) \quad \mbox{such that} \label{cae} \feq
$$ {d u \over d t} = \PPP(u,u)~, \qquad u(0) = u_0~. $$
Let us choose an order $N \in \{0,1,2,...\}$ and consider as an approximate
solution for \rref{cae} a polynomial of degree $N$ in time, of the form
\beq u^N : [0, +\infty) \vain \HM{n+1}~,
\qquad t \mapsto u^N(t) := \sum_{j=0}^N u_j t^j
\qquad (u_j \in \HM{n+1}~\mbox{for all $j$})~. \label{defun} \feq
Here $u_0$ is the initial datum, and $u_j$ is to be determined
for $j=1,...,N$.
\begin{proposition}
(i) Let $u^N$ be as in \rref{defun}. The datum and differential errors of $u^N$ are
\beq u^N(0) - u_0 = 0~; \label{daer} \feq
\beq {d u^N \over d t}(t) - \PPP(u^N, u^N)(t) \label{eun} \feq
$$ =
\sum_{j=0}^{N-1} \Big[ (j+1) u_{j+1} - \sum_{\ell=0}^j \PPP(u_\ell, u_{j-\ell}) \Big] t^{j}
- \sum_{j=N}^{2 N} \Big[ \sum_{\ell=j-N}^N \PPP(u_\ell, u_{j - \ell}) \Big] t^j~. $$
(ii) In particular, assume
\beq u_{j+1} = {1 \over j+1} \sum_{\ell=0}^j \PPP(u_{\ell}, u_{j-\ell})
\quad \mbox{for $j=0,...,N-1$}~; \label{recur} \feq
then
\beq {d u^N \over d t}(t) - \PPP(u^N, u^N)(t) =
- \sum_{j=N}^{2 N} \Big[ \sum_{\ell=j-N}^N \PPP(u_\ell, u_{j - \ell}) \Big] t^j =
O(t^N)~\mbox{for $t \vain 0$}~. \label{eune} \feq
(iii) If \rref{recur} is used to define recursively $u_1,...,u_N$,
it produces a sequence of elements of $\HM{n+1}$ under
the condition $u_0 \in \HM{n+1+N}$. More precisely,
from $u_0 \in \HM{n+1+N}$ it follows $u_j \in \HM{n+1+N-j} \subset \HM{n+1}$
for $j=1,...,N$. \par\noindent
(iv) Let $u_0 \in \HM{n+N+1}$ and use \rref{recur} to define
$u_j$ for $j=1,...,N$. Then
\beq \| {d u^N \over d t}(t) - \PPP(u^N, u^N)(t) \|_n \leqs \ep_n(t)
\qquad \mbox{~~for $t \in [0,+\infty)$}~, \label{est} \feq
\beq \ep_n(t) := K_n \sum_{j=N}^{2 N} \Big[ \sum_{\ell=j-N}^N \| u_\ell \|_n \| u_{j - \ell}\|_{n+1} \Big] t^j~
\mbox{for $t \in [0,+\infty)$}~.
\label{est2} \feq
\end{proposition}
\begin{proof}
(i) \rref{daer} is obvious; let us prove \rref{eun}.
To this purpose, we note that
$$ {d u^N \over d t} - \PPP(u^N, u^N)  =
{d \over d t} \big( \sum_{\ell=0}^N u_\ell t^\ell \big)
- \PPP\big(\sum_{\ell=0}^N u_\ell t^\ell,
\sum_{h=0}^N u_h t^h\big) $$
$$ = \sum_{\ell=1}^N \ell u_\ell t^{\ell-1} - \sum_{\ell,h=0}^N \PPP(u_\ell, u_h)
t^{\ell+h}
= \sum_{j=0}^{N-1} (j+1) u_{j+1} t^{j} - \sum_{j=0}^{2 N}
\big[ \sum_{(\ell,h) \in I_{N j}} \PPP(u_\ell, u_h) \big] t^j~, $$
$$ I_{N j} := \{ (\ell,h) \in \{0,...,N\}^2~|~\ell+h= j \}~. $$
One easily checks that
$$ j \in \{0,...,N-1\}~\Rightarrow~
I_{N j} = \{ (\ell,j-\ell)~|~\ell \in \{0,...,j\} \}~, $$
$$ j \in \{N,...,2 N\}~\Rightarrow~
I_{N j} = \{ (\ell,j-\ell)~|~\ell \in \{j-N,...,N\} \}~; $$
this readily yields the thesis \rref{eun}. \parn
(ii) Obvious. \parn
(iii) Let $u_0 \in \HM{n+1+N}$ and define $u_1,...,u_N$
via the recursion relation \rref{recur}. Then
$u_1 = \PPP(u_0, u_0) \in \HM{n+N}$, $u_2 = (1/2) \PPP(u_0, u_1) + (1/2) \PPP(u_1, u_0)
\in \HM{n+N-1}$, etc. \, . \parn
(iv) Eq. \rref{eune} implies
$\| (d u^N/d t)(t) - \PPP(u^N, u^N)(t) \|_n$ $\leqs
\sum_{j=N}^{2 N} \Big[ \sum_{\ell=j-N}^N \| \PPP(u_\ell, u_{j - \ell}) \|_n \Big] t^j$. On
the other hand Eq. \rref{basic} gives
$\| \PPP(u_\ell, u_{j - \ell}) \|_n \leqs K_n  \| u_\ell \|_n \| u_{j - \ell}\|_{n+1}$,
whence the thesis \rref{est} \rref{est2}.
\end{proof}
\vfill \eject \noindent
\section{A special case of the previous framework: the Euler equations
on $\Tt$, with the Behr-Ne$\check{\mbox{c}}$as-Wu
initial datum.}
In this section we consider the Euler Cauchy problem
\rref{cae} with space dimension and Sobolev order
\beq d = 3~, \qquad n = 3~; \feq
the initial datum is
\beq u_0 := \sum_{k = \pm a, \pm b, \pm c} u_{0 k} e_k~, \label{unec} \feq
$$ a := (1,1,0),~~ b := (1,0,1),~~ c := (0,1,1)~; $$
$$ u_{0, \pm a} := (2 \pi)^{3/2} (1,-1,0), \quad u_{0, \pm b}
:= (2 \pi)^{3/2} (1,0,-1), \quad u_{0, \pm c} := (2 \pi)^{3/2} (0,1,-1) $$
(of course, being a Fourier polynomial, $u_0$ belongs to $\HM{m}$
for each $m \in \reali$). The above initial
datum is considered by Behr, Ne$\check{\mbox{c}}$as and
Wu in \cite{Nec}; it is analyzed with a similar attitude in
\cite{bnw} (and, from a different viewpoint, in \cite{appeul}).
In both papers \cite{Nec} \cite{bnw}, attention
is fixed on the function $u^N(t) = \sum_{j=0}^N u_j t^j$ for
a rather large value of $N$,
where the $u_j$'s are determined for $j=1,...,N$ by the recursion relation
\rref{recur}. The $u_j$'s are Fourier polynomials
and can be calculated exactly by computer algebra
methods; such computations
are performed in \cite{Nec} for $N=35$,
and in \cite{bnw} up to $N=52$ (using, respectively,
the $C^{++}$ and the Python languages).
\par
The Python program of \cite{bnw} gives exact
expressions for the $u_{j}$'s, whose Fourier components
are rational (up to factors $(2 \pi)^{3/2}$);
for large $j$, these expressions are terribly complicated.
Here, to give a partial illustration of such Python computations
we consider the Fourier components $u^{52}_k(t)$ for $k = (1,1,0)$
and $k = (0,0,2)$, and report the graphs of
the functions $t \mapsto |u^{52}_k(t)|$ for these wave vectors:
see Figures 1 and 2. \par
In both papers \cite{Nec} \cite{bnw}, computations are used to get hints
about $\lim_{N \vain + \infty} u^N$, giving
the exact solution of the Euler Cauchy problem
on the time interval where the limit exists;
however the statements of \cite{Nec} \cite{bnw} rely on the assumption
that certain facts on the $N \vain + \infty$ limit
can be extrapolated from $u^{35}$ or $u^{52}$.
In particular \cite{Nec} makes the conjecture,
disputed in \cite{bnw}, that the solution
of the Euler Cauchy problem blows up
for $t \vain \tau^{-}$, with $\tau \simeq 0.32$
\par
In the present work we make no conjecture or extrapolation
about the $N \vain + \infty$ limit and just consider
the function $u^{52}$ of \cite{bnw} according to
the general framework of approximate solutions
and control inequalities. This
approach produces: \par\noindent
(i) a rigorous lower bound on the interval of existence
of the exact solution $u$ of the ($d=3, n= 3$) Cauchy problem \rref{cae}; \par\noindent
(ii) a bound on $\| u(t) - u^{52}(t) \|_3$. \par
To get these results we regard $u^{52}$ as an approximate solution of \rref{cae},
using the tautological datum error and growth estimators
\beq \delta_3 := 0~; \quad \Dd_3(t) := \| u^{52}(t) \|_3~,~~\Dd_{4}(t) := \| u^{52}(t) \|_{4}
~~\mbox{for $t \in [0,+\infty)$} \label{est1} \feq
(concerning $\delta_3$, we recall that $u^{52}(0) - u_0 = 0$). For $m=3,4$
one has $\Dd_m(t) = (2 \pi)^{3/2} [\sum_{j=0}^{52} d_{m j} t^{2 j}]^{1/2}$
where the $d_{m j}$ are rational coefficients; the Python program employed
for our work \cite{bnw} computes exactly these coefficients. For
$m=3$ these coefficients are reported in \cite{bnw}, in a $16$-digits decimal
representation (see Eq. (5.12) of \cite{bnw}, not containing
the factor $(2 \pi)^{3/2}$ due to a different normalization
of the norm $\|~ \|_3$); we have no room to report here the coefficients
of the $m=4$ case. Figures 3 and 4 contain the graphs
of the functions $t \mapsto \Dd_3(t), \Dd_4(t)$. \par
Let us pass to the differential error estimator for $u^{52}$; we use for it
the function $\ep_3$ defined by
\rref{est2} with $n=3$ and $K_3 = 0.323$, see \rref{k3g3}.
$\ep_3$ is computed exactly by our Python program; again,
the explicit expression is too complicated to be reported.
(The tautological error estimator
$\ep^{*}_3(t) :=$
$\| (d u^{52}/d t)(t) - \PPP(u^{52}, u^{52})(t) \|_3$
$= \| \sum_{j=52}^{104} \Big[ \sum_{\ell=j-52}^{52} \PPP(u_\ell, u_{j - \ell}) \Big] t^j \|_3$
is more accurate, but it has an even more complicated expression;
its calculation by computer algebra is too expensive.) \par
For the graph of $\ep_3$ and some information on its numerical
values, see Figure 5 and its caption.
With the previous ingredients, we build the following ``control Cauchy problem'':
find $\Rr_3$ such that
\beq \Rr_3 \in C^1([0, \Tc), \reali),~~
{d \Rr_3 \over d t} = (G_3 \Dd_3 + K_3 \Dd_{4}) \Rr_3 + G_3 \Rr^2_3 + \ep_3,~~
\Rr_3(0) = 0 \label{contcau} \feq
($G_3 = 0.438$, see again \rref{k3g3}). This control problem
has a unique maximal solution $\Rr_3$,
which is strictly increasing and thus positive for $t \in (0,\Tc)$.
Of course, this $\Rr_3$ fulfils as equalities Eqs.
\rref{cont1} \rref{cont2} (with $\nu=0$). \par
Once we have $\Rr_3 : [0,\Tc) \vain [0,+\infty)$, due to Proposition
\ref{main} we can grant that: \parn
(i) The maximal solution $u$ of the ($n=3$) Euler Cauchy problem
\rref{cae} is defined on an interval $[0,T)$ with $T \geqs \Tc$; \parn
(ii) It is
\beq \| u(t) - u^{52}(t) \|_3 \leqs \Rr_3(t) \qquad \mbox{for $t \in [0, \Tc)$}~. \label{itis} \feq
The function $\Rr_3$ can be determined numerically by a cheap
computation using any package for ODEs, e.g. Mathematica (the result is
reliable, since \rref{contcau} is the Cauchy problem for a simple
ODE in one dimension). This numerical computation indicates that
the (maximal) domain of $\Rr_3$ is $[0, \Tc)$, with
\beq \Tc = 0.242... ; \label{tc} \feq
After having been extremely small for most of the time
between $0$ and $\Tc$,
$\Rr_3(t)$ diverges abruptly for $t \vain \Tc^{-}$; for
the graph of this function and some information
on its numerical values, see Figure 6 and its caption.
Due to \rref{tc}, we can grant that the solution $u$
of the Euler Cauchy problem \rref{cau} exists on
a time interval of length $T \geqs 0.242$ (this is
four times larger than the lower bound on $T$ obtained in
\cite{appeul} using a Galerkin approximate solution). \par
Eq. \rref{itis} and the previously described
behavior of $\Rr_3$ ensure that $u^{52}(t)$
approximates with extreme precision $u(t)$
on most of the time interval $[0, \Tc)$.
We remark that \rref{itis} can be used to infer
other interesting estimates about $u - u^{52}$, e.g.,
\beq | u_k(t) - u^{52}_k(t) | \leqs {\Rr_3(t) \over |k|^3}
\qquad \mbox{for $k \in \Zt \setminus \{0 \}$, $t \in [0, \Tc)$}~;
\feq
this follows from \rref{itis} and from
the elementary inequality $| v_k | \leqs \| v \|_3/|k|^3$, holding
for all $v \in \HM{3}$ and $k \in \Zt \setminus \{ 0 \}$
(recall that $\| v \|^2_3 = \sum_{k \in \Zt \setminus \{ 0 \}}
|k|^6 | v_k |^2$). \par
\begin{figure}
\hskip -3cm \noindent
\parbox{3in}{
\includegraphics[
height=2.0in,
width=2.8in
]%
{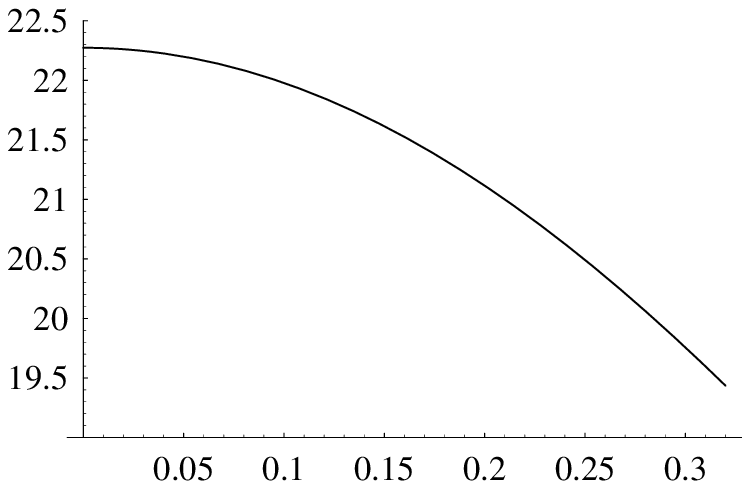}%
\par
{
{
{\textbf{Figure 1.~} Plot of $|u^{52}_{(1,1,0)}(t)|$
for $t \in [0,0.32]$.
}}
\par}
\label{f1a}
}
\hskip 0.3cm
\parbox{3in}{
\includegraphics[
height=2.0in,
width=2.8in
]%
{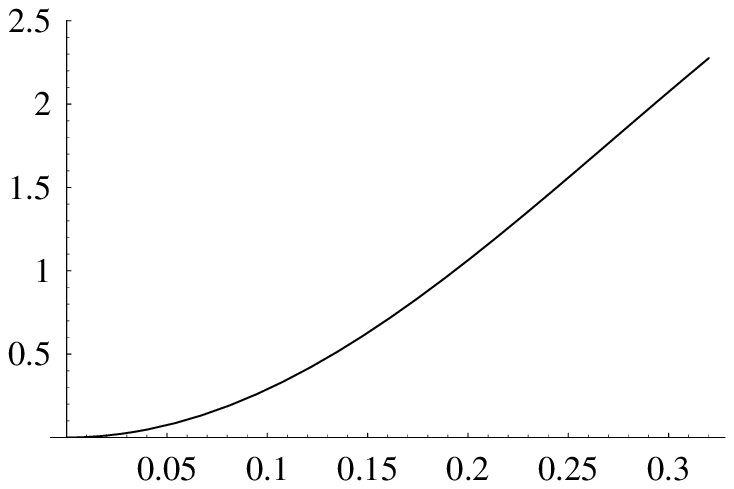}%
\par
{
{
{\textbf{Figure 2.~} Plot of $|u^{52}_{(0,0,2)}(t)|$
for $t \in [0,0.32]$.
}}
\par}
\label{f1d}
}
\vskip 0.3cm  \noindent
{~}
\hskip -3.2cm \noindent
\parbox{3in}{
\includegraphics[
height=2.0in,
width=2.8in
]%
{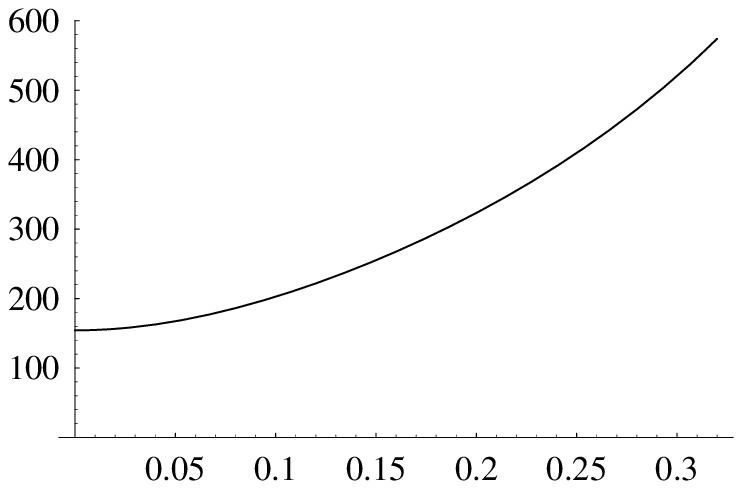}%
\par
{
{
{\textbf{Figure 3.~} Plot of $\Dd_3(t)$ for
$t \in [0,0.32]$.
}}
\par}
\label{f1b}
}
\hskip 0.3cm
\parbox{3in}{
\includegraphics[
height=2.0in,
width=2.8in
]%
{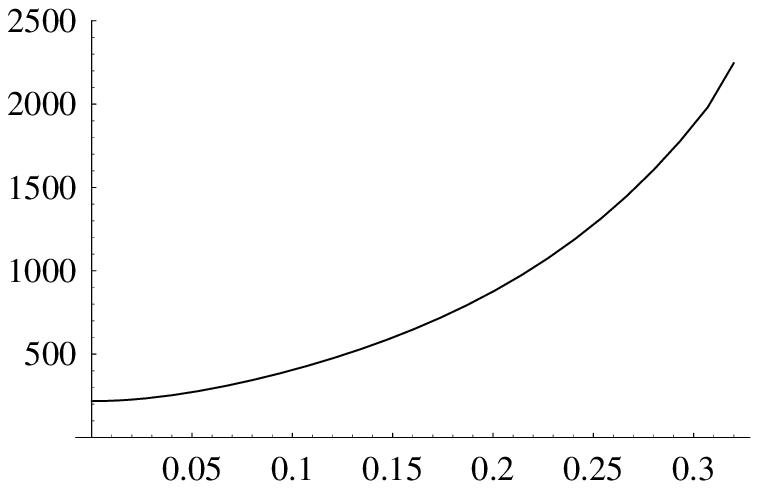}%
\par
{
{
{\textbf{Figure 4.~} Plot of $\Dd_4(t)$ for
$t \in [0,0.32]$.
}}
\par}
\label{f1e}
}
\vskip 0.5cm \noindent
{~}
\hskip -3.2cm
\parbox{3in}{
\includegraphics[
height=2.0in,
width=2.8in
]%
{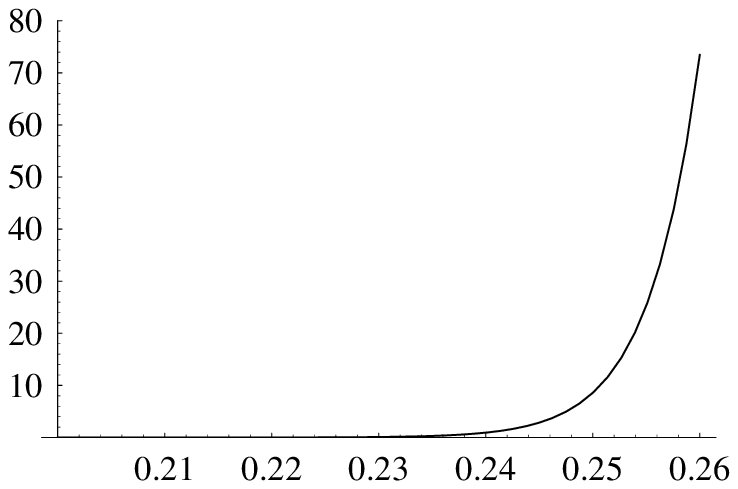}%
\par
{
{
{\textbf{Figure 5.~} Plot of $\epsilon_3(t)$
for $t \in [0.20, 0.26]$. One has:
$\epsilon_3(t) < 10^{-20}$ for $t \in [0,0.10]$;
$\epsilon_3(t) < 10^{-4}$ for $t \in (0.10, 0.20]$;
$\epsilon_3(t) < 10^{-3}$ for $t \in (0.20, 0.21]$;
$\epsilon_3(t) < 8.6 \times 10^{-3}$ for $t \in (0.21, 0.22]$;
$\epsilon(t) < 0.094$ for $t \in (0.22, 0.23]$;
$\epsilon(t) < 0.93$ for $t \in (0.23, 0.24]$,
$\epsilon_3(t) < 8.6$ for $t \in (0.24, 0.25]$;
$\epsilon_3(t) < 74$ for $t \in (0.25, 0.26]$.
}}
\par}
\label{f1c}
}
\hskip 0.3cm
\parbox{3in}{
\vskip -0.9cm \noindent
\includegraphics[
height=2.0in,
width=2.8in
]%
{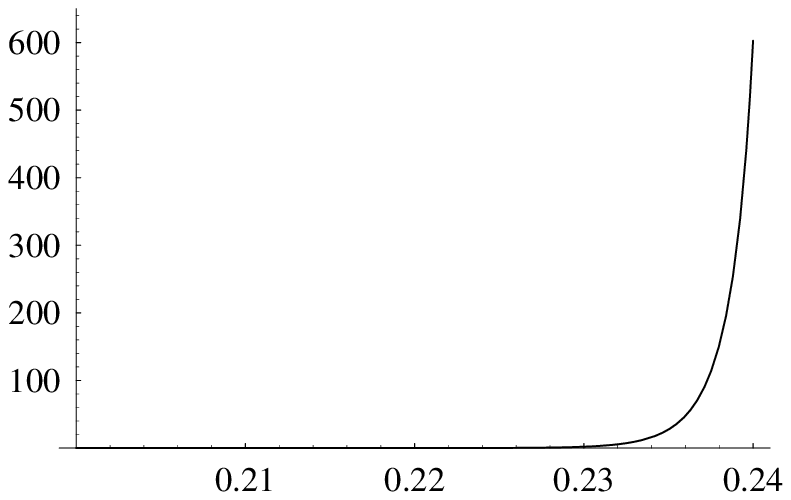}%
\par
{
{
{\textbf{Figure 6.~} Plot of $\Rr_3(t)$
for $t \in [0.20, 0.24]$. One has:
$\Rr_3(t) < 2 \times 10^{-6}$ for
$t \in [0,0.20]$; $\Rr_3(t) < 1.2 \times 10^{-4}$
for $t \in (0.20, 0.21]$; $\Rr_3(t) < 0.013$ for
$t \in (0.21, 0.22]$; $\Rr_3(t) < 2$ for
$t \in (0.22, 0.23]$; $\Rr_3(t) < 610$ for $t \in (0.23, 0.24]$.

}}
\par}
\label{f1f}
}
\end{figure}
\vfill \eject \noindent
\textbf{Acknowledgments.} This work has been partially supported by
by INdAM, INFN and by MIUR, PRIN 2010
Research Project "Geometrical methods in the theory of nonlinear waves and applications".
\vskip 0.4cm \noindent
{~}

\vskip 0.8cm \noindent
\end{document}